\documentclass[11pt,a4paper]{article}

\usepackage{amsmath,amsfonts,amssymb,amsthm}
\usepackage{mathbbol}
\usepackage{amsthm}
\usepackage{graphicx,color}
\usepackage{boxedminipage}
\usepackage[linesnumbered, boxed, algosection,noline]{algorithm2e}
\usepackage{framed}
\usepackage{thmtools}
\usepackage{thm-restate}
\usepackage{xspace}
\usepackage{vmargin}
\setmarginsrb{1in}{1in}{1in}{1in}{0pt}{0pt}{0pt}{6mm}
\usepackage{todonotes}
 \usepackage[pdftex, plainpages = false, pdfpagelabels, 
                 bookmarks=false,
                 bookmarksopen = true,
                 bookmarksnumbered = true,
                 breaklinks = true,
                 linktocpage,
                 pagebackref,
                 colorlinks = true,  
                 linkcolor = blue,
                 urlcolor  = blue,
                 citecolor = red,
                 anchorcolor = green,
                 hyperindex = true,
                 hyperfigures
                 ]{hyperref} 
 \usepackage{xifthen}
 \usepackage{tabularx}

 \usetikzlibrary{calc} 

\newtheorem{theorem}{Theorem}
\newtheorem{corollary}{Corollary}

\newtheorem{question}{Question}

\newtheorem{lemma}{Lemma}
\newtheorem{claim-n}{Claim}[section]

\DeclareMathOperator*{\argmax}{arg\,max}
\DeclareMathOperator*{\argmin}{arg\,min}
\newcommand{\opt}{\textsc{opt}}
\newcommand{\cost}{\textsc{cost}}

\newcommand{\m}[1]{\mathcal{#1}}
\newcommand{\w}{\omega}

\title{Present-Biased Optimization\thanks{A preliminary version of this paper is accepted for AAAI 2021. 
The research received funding from  the Research Council of Norway via the project ``MULTIVAL" (grant no. 263317) and from the French National Research Agency (ANR) via the project ``DESCARTES''. }}

\author {
    Fedor V. Fomin\thanks{Department of Informatics, University of Bergen, Norway}
    \and
    Pierre Fraigniaud\thanks{IRIF, Universit\'e de Paris and CNRS, France}\addtocounter{footnote}{-2}   
     \and
    Petr A. Golovach\footnotemark{}
}

\begin{document}

\maketitle

\begin{abstract}
This paper explores the behavior of \emph{present-biased} agents, that is, agents who erroneously anticipate the costs of future actions compared to their real costs. Specifically, the paper extends the original framework proposed by Akerlof (1991)  for studying various aspects of human behavior related to time-inconsistent planning, including procrastination, and abandonment, as well as the elegant graph-theoretic model encapsulating this framework recently proposed by Kleinberg and Oren (2014). The benefit of this extension is twofold. First, it enables to perform fine grained analysis of the behavior of present-biased agents depending on the \emph{optimisation task} they have to perform. In particular, we study covering tasks vs. hitting tasks, and show that the ratio between the cost of the solutions computed by present-biased agents and the cost of the optimal solutions may differ significantly depending on the problem constraints. 
Second, our extension enables to study not only underestimation of future costs, coupled with minimization problems, but also all combinations of minimization/maximization, and underestimation/overestimation. We study the four scenarios, and we establish upper bounds on the cost ratio for three of them (the cost ratio for the original scenario was known to be unbounded),  providing a complete global picture of the behavior of present-biased agents, as far as optimisation tasks are concerned. 
\end{abstract}

\section{Introduction}
  \emph{Present bias} is the term used in behavioral economics to describe 
the gap between the anticipated costs  of future actions and their real costs.  
A simple mathematical model of  
    present bias was suggested by Akerlof~\cite{Akerlof91}. In this model 
the cost of an action that will be  perceived in the future is assumed to be $\beta$ times smaller than its actual cost,  for some constant  $\beta <1$, called the \emph{degree of present bias}.  The model was used for studying various aspects of human behavior related to \emph{time-inconsistent} planning, including \emph{procrastination}, and \emph{abandonment}. 

Kleinberg and Oren~\cite{KleinbergO14,KleinbergO18} introduced an elegant  graph-theoretic model  encapsulating Akerlof's model.  
 The approach of Kleinberg and Oren  is based on analyzing how an agent navigates from a source~$s$ to a target~$t$ in a directed  edge-weighted graph~$G$, called \emph{task graph}. At any step, the agent chooses the next edge to traverse  from the current vertex~$v$ thanks to an estimation of the length of the shortest path from~$v$ to~$t$ passing through each edge outgoing from~$v$. A crucial characteristic of the model is that the estimation of the path lengths is \emph{present-biased}. More specifically, the model of Kleinberg and Oren includes a positive parameter $\beta<1$, the degree of present bias,  and the length of a path $x_0,\dots,x_k$ from $x_0=v$ to $x_k=t$ in~$G$ is evaluated as 
 $
 \w_0+ \beta  \, \sum_{i=1}^{k-1}\w_i
 $
  where $\w_i$ denotes the weight of edge $(x_i,x_{i+1})$, for every $i\in\{0,\dots,k-1\}$. As a result, the agent may choose an outgoing edge that is not on any shortest path from~$v$ to~$t$, modeling \emph{procrastination} by underestimating the cost of future actions to be performed whenever acting now in some given way. With this effect cumulating along its way from~$s$ to~$t$, the agent may significantly diverge from shortest $s$-$t$ paths, which demonstrates the negative impact of procrastination. Moreover, 
 the \emph{cost ratio}, which is  the ratio between the cost of the path traversed by an agent with present bias and the cost of a shortest path, could be arbitrarily large. 
   An illustrating example is depicted on Fig.~\ref{fig:Akerlof}, borrowed from~\cite{KleinbergO18}, and originally due to Akerlof~\cite{Akerlof91}. Among many results, Kleinberg and Oren showed that  any graph in which a present-biased agent incurs significantly more cost than an optimal agent must contain a large specific structure as a minor. This structure, called \emph{procrastination structure}, is specifically the one depicted on Fig.~\ref{fig:Akerlof}. 

\begin{figure}[h]
\centering
\includegraphics{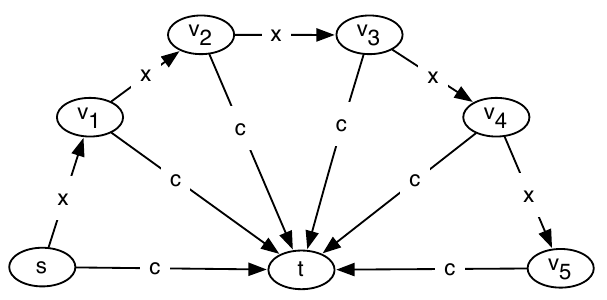}
\caption{\sl Procrastination structure as displayed in~\cite{KleinbergO18}; Assuming $\mathsf{x}+\beta \mathsf{c} < \mathsf{c}$, the path followed by the agent is $\mathsf{s},\mathsf{v}_1,\dots,\mathsf{v}_5,\mathsf{t}$; The ratio between the length of the path followed by the agent and the shortest $\mathsf{s}$-$\mathsf{t}$ path can be made arbitrarily large by adding more nodes $\mathsf{v}_k$ with $k\geq 5$.}
\label{fig:Akerlof}
\end{figure}

In this paper, we are interested in understanding what kind of \emph{tasks} performed by the agent result in large cost ratio. 
 Let us take the concrete example of an agent willing to acquire the knowledge of a set of scientific concepts, by reading books. Each book covers a certain number of these concepts, and the agent's objective is to read as few books as possible. More generally, each book could also be weighted according to, say, its accessibility to a general reader, or its length. The agent's objective is then to read a collection of books with minimum total weight. Both the weight and the collection of concepts covered by each book are known to the agent a priori. This scenario is obviously an instance of the (weighted) \emph{set-cover} problem. Let us assume, for simplicity, that the agent has access to a \emph{time-biased} oracle providing it with the following information. Given the subset of concepts already acquired by the agent when it queries the oracle, the latter returns to the agent a set $\{b_0,\dots,b_{k-1}\}$ of books minimizing  $\w_0+ \beta  \, \sum_{i=1}^{k-1}\w_i$ among all sets of books covering the concepts yet to be acquired by the agent, where $\w_0\leq \w_1 \leq \dots \leq \w_{k-1}$ are the respective weights of the books $b_0,\dots,b_{k-1}$. This corresponds to the  procrastination scenario in which the agent picks the easiest book to read now, and underestimates the cost of reading the remaining books later. Then the agent moves on by reading~$b_0$, and querying the oracle for figuring out the next book to read for covering the remaining uncovered concepts after having read book~$b_0$. 

The question is: by how much the agent eventually diverges from the optimal set of books to be read? This set-cover example fits with the framework of Kleinberg and Oren, by defining the vertex set of the task graph as the set of all subsets of concepts, and placing an edge $(u,v)$  of weight~$\w$ from~$u$ to~$v$ if there exists a book~$b$ of weight~$\w$ such that~$v$ is the union of~$u$ and the concepts covered by~$b$. In this setting, the agent needs to move from the source $s=\varnothing$ to the target~$t$ corresponding to the set of all the concepts to be acquired by the agent. Under this setting, the question can be reformulated as: under which circumstances the set-cover problem yields a large  cost ratio?

More generally, let us consider a  \emph{minimization} problem where, for every feasible solution $S$ of every instance of the problem, the cost~$c(S)$ of $S$ can be expressed as  $c(S)=\sum_{x\in S}\w(x)$ for some weight function~$\w$. This includes, e.g., set-cover, min-cut, minimum dominating set, feedback-vertex set, etc. We then define the biased cost $c_\beta$ as 
\begin{equation}\label{eq:cost-beta}
c_\beta(S)=\w(x^*)+\beta \, c(S\smallsetminus \{x^*\}),
\end{equation}
where $x^*=\argmin_{x\in S}\w(x)$. Given an instance $I$ of the minimization problem at hand, the agent aims at finding a feasible solution $S\in I$ minimizing $c(S)$. It does so using the following present-biased planning, where $I_0=I$.

\paragraph{\bf Minimization scenario:} 

For $k\geq 0$, given an instance~$I_k$, the agent computes the feasible solution~$S_k$ with minimum cost $c_\beta(S_k)$ among all feasible solutions for~$I_k$. Let $x^*_k = \argmin_{x\in S_k}w(x)$. The agent stops whenever  
$
\{x^*_0,x^*_1,\dots,x^*_k\}
$
is a feasible solution for~$I$. Otherwise, the agent  moves to $I_{k+1}=I_k\smallsetminus x^*_k$, that is, the instance obtained from $I_k$ when one assumes~$x^*_k$ selected in the solution. 

\medskip

This general scenario is indeed captured by the Kleinberg and Oren model, by defining the vertex set of the graph task graph as the set of all ``sub-instances'' of the instance~$I$ at hand, and placing an edge $(u,v)$  of weight~$w$ from~$u$ to~$v$ if there exists an element $x$ of weight~$\w$ such that~$v$ results from~$u$ by adding~$x$ to the current solution. The issue is to analyze how far the solution computed by the present-biased agent is from the optimal solution. The first question addressed in this paper is therefore the following. 

\begin{question}\label{prob:procrastination}
For which minimization tasks a large cost ratio 
may appear? 
\end{question}
In the models of  Akerlof~\cite{Akerlof91} and  Kleinberg and Oren~\cite{KleinbergO18} the degree $\beta$ of present bias  is assumed to be less than one. However, there are natural situations where underestimating the future costs does not hold. For example, in their influential paper, Loewenstein,  O'Donoghue, and Rabin~\cite{loewenstein2003projection} gave a number of  examples  from a variety of domains demonstrating the prevalence of \emph{projection bias}. In particular, they reported an experiment by Jepson, Loewenstein, and Ubel~\cite{JepsonLoewensteinUbel2001} who ``asked people waiting for a kidney transplant to predict what their quality of life would be one year later if they did or did not receive a transplant, and then asked those same people one year later to report their quality of life. Patients who received transplants predicted a higher quality of life than they ended up reporting, and those who did not predicted a lower quality of life than they ended up reporting''. In other words, there are situations in which people may also \emph{overestimate} the future costs. In the model of Kleinberg and Oren~\cite{KleinbergO18} overestimation bias corresponds to the situation of putting the  degree of present bias $\beta>1$.
This brings us to the second question. 

\begin{question}\label{prob:procrastination2}
Could a large cost ratio appear for minimization problems when the  degree of present bias $\beta$ is more than $1$?
\end{question}

Reformulating the analysis of procrastination, as stated in Question~\ref{prob:procrastination}, provides inspiration for tackling related problems. 
As a matter of fact, under the framework of  Kleinberg and Oren, procrastination is a priori associated to \emph{minimization} problems. We also investigate \emph{maximization} problems, in which a present-biased agent aims at, say, maximizing its revenue by making a sequence of actions, each providing some immediate gain that the agent maximizes while underestimating the incomes resulting from future actions. As a concrete example, let us consider an instance of Knapsack. The agent constructs a solution gradually by picking the item $x_0$ of highest value $\w(x_0)$ in a feasible set $\{x_0,\dots,x_{k-1}\}$ of items that is maximizing $\w(x_0)+ \beta  \, \sum_{i=1}^{k-1}\w(x_i)$ for the current sub-instance of Knapsack. In general, given an instance $I$ of a maximisation problem, we assume that the agent applies the following present-biased planning, with $I_0=I$:

\paragraph{\bf Maximization scenario:}

Given an instance~$I_k$ for $k\geq 0$, the agent computes the feasible solution~$S_k$ with maximum cost $c_\beta(S_k)$ among all feasible solutions for~$I_k$ --- where the definition of~$x^*$ in Eq.~\eqref{eq:cost-beta} is replaced by $x^*=\argmax_{x\in S}w(x)$. With $x^*_k = \argmax_{x\in S_k}w(x)$, the agent stops whenever  $\{x^*_0$, $x^*_1$, $\dots$, $x^*_k\}$ is an inclusion-wise maximal feasible solution for~$I$, and moves to $I_{k+1}=I_k\smallsetminus x^*_k$ otherwise. 

\medskip

We are interested in analyzing how far the solution computed by the present-biased agent is from the optimal solution. More generally even, we aim at revisiting time-inconsistent planning by considering both cases $\beta<1$ and $\beta>1$, that is, not only  scenarios in which the agent underestimates the cost of future actions, but also scenarios in which the agent  \emph{overestimates} the cost of future actions.  The last, more general question addressed in this paper is therefore the following. 

\begin{question}\label{prob:general}
For which optimization tasks, and for which time-inconsistency planning (underestimation, or overestimation of the future actions), the  solutions computed by a present-biased agent are far from optimal, and for which they are close?
\end{question}

For all these problems, we study the \emph{cost ratio} $\varrho=\frac{c(S)}{\opt}$ (resp., $\varrho=\frac{\opt}{c(S)}$) where~$S$ is the solution returned by the present-biased agent, and $\opt=c(S_{\opt})$ is the cost of an optimal solution for the same instance of the considered minimization (resp., maximization) problem. 

\subsection{Our Results}

Focussing on agents aiming at solving tasks, and not just on agents aiming at reaching targets in abstract graphs, as in the generic model in~\cite{KleinbergO18}, allows us not only to refine the worst-case analysis of present-biased agents, but also to extend this analysis to scenarios corresponding to overestimating the future costs to be incurred by the agents (by setting the degree $\beta$ of present bias larger than~1), and to maximisation problems. 

\paragraph{\bf Minimization \& underestimation.}

In the original setting of minimization problems, with underestimation of future costs (i.e., $\beta<1$), we show that the cost ratio~$\varrho$ of an agent performing~$k$ steps, that is, computes a feasible solution $\{x^*_1,\dots,x^*_k\}$, satisfies $\varrho\leq k$. This is in contrast to the general model in~\cite{KleinbergO18}, in which an agent can incur a cost ratio exponential in~$k$ when returning a $k$-edge path from the source to the target.  Hence, in particular, our minimization scenarios do not produce the worst cases examples constructed in~\cite{KleinbergO18}, i.e., obtained by considering travels from sources to targets in arbitrary weighted graphs. 

On the other hand, we also show that a ``minor structure'' bearing similarities with the one identified in~\cite{KleinbergO18} can be identified. Namely, if an agent incurs a large cost ratio, then the minimization problem addressed by the agent includes a large instance of a specific form of minimization problem. 

\paragraph{\bf Min/maximization \& under/overestimation.}

Interestingly, the original setting of minimization problems, with underestimation of future costs, is far from reflecting the whole nature of the behavior of present-biased agents. Indeed, while minimization problems with underestimation of future costs may result in unbounded cost ratios,  the worst-case cost ratios corresponding to the three other settings can be upper bounded, some by a constant independent of the task at hand. Specifically, we show that: 
\begin{itemize}
\item For any minimization problem with $\beta>1$, the cost ratio is at most~$\beta$; 
\item For any maximization problem with $\beta<1$, the cost ratio is at most~$\frac{1}{\beta}$; 
\item For any maximization problem with $\beta>1$, the cost ratio is at most~$\beta^c$, where $c\leq\opt$ is the cost of a solution constructed by the agent. 
\end{itemize}
Our results are summarized in Table~\ref{tab:summary2}. 

\begin{table}[ht]
\definecolor{thmgray}{gray}{0.3}
\begin{center}
\begin{tabular}{r|c|c|}
 			& minimization & maximization \\
\hline
 $\beta<1$ 	& $\infty$~\cite{KleinbergO18} & $   1/\beta$ \textcolor{thmgray}{[Thm~\ref{thm:min-max-beta}(i)]} \\
 \hline 
$\beta>1$ 	& $ \beta$ \textcolor{thmgray}{[Thm~\ref{thm:min-betamoreone}]}& $\left(1+\log\beta\right)\frac{\opt}{\log \opt}$ \textcolor{thmgray}{[Cor~\ref{cor:upper}]} \\
\hline
\end{tabular}
\end{center}
\caption{\sl Upper bounds on the worst case ratio between the solution cost returned by the present-biased agent  and the optimal solution~$\opt$. The symbol $\infty$ means that the cost ratio can be arbitrarily large, independently of the values of $\beta$, and $\opt$. }
\label{tab:summary2}
\end{table}

Let us remark that, for minimization problems with $\beta>1$, as well as for maximization problems with $\beta<1$, we have that the cost ratio is bounded by a constant. However, for  maximization problems with $\beta>1$,   the cost ratio can be   exponential in the cost of the computed solution. We show that this exponential upper bound  is essentially tight.

\paragraph{\bf Approximated evaluations.}

Actually, in many settings, discrete optimization problems are hard. Therefore, for evaluating the best feasible solution according to the biased cost function $c_\beta$, an agent may have to solve computationally  intractable problems. Thus, in a more realistic scenario, we assume that, instead of computing an 
optimal solution for $c_\beta$ at every step, the agent computes an $\alpha$-approximate solution. 

\paragraph{\bf Fine-grained analysis.}

In contrast to the general model in~\cite{KleinbergO18}, the refined model of this paper enables fine-grain analysis of the agents' strategies, that is, it enables identifying different behaviors of the agents as a function of the considered  optimisation problems. Specifically,  there are natural minimization problems for which specific bounds on the cost ratio can be established.    

To illustrate the interest of focusing on optimisation tasks, we study two tasks in detail, namely \emph{set-cover} and \emph{hitting set}, and show that they appear to behave quite differently. For set-cover, we show that the cost ratio is at most $d\cdot\opt$, where $d$ is the maximum size of the sets. For hitting set, we show that the cost ratio is at most $d!\,(\frac{1}{\beta}\opt)^{d}$, again for $d$ equal to the maximum size of the sets. 
This latter result is using the sunflower lemma of Erd\H{o}s and Rado~\cite{ErdosR60}.

Finally, we identify a simple restriction of the agent's strategy, which guarantees that the cost of the solution computed by the agent is not more than~$\beta$ times the cost of an optimal solution. 

\subsection{Related Work}

Our work is directly inspired by the aforementioned contribution of Kleinberg and Oren~\cite{KleinbergO14}, which was itself motivated by the earlier work by Akerlof~\cite{Akerlof91}. We refer to~\cite{KleinbergO14,KleinbergO18} for a survey of earlier work on time-inconsistent planning, with connections to procrastination, abandonment, and choice reduction. Hereafter, we discuss solely~\cite{KleinbergO14}, and the subsequent work. Using their graph-theoretic framework, Kleinberg and Oren reasoned about time-inconsistency effects. In particular, they provided a characterization of the graphs yielding the worst-case cost-ratio,
and they showed that, despite the fact that the degree $\beta$ of present bias can take all possible values in $[0,1]$, it remains that, for any given digraph,  the collection of distinct $s$-$t$ paths computed  by present-biased agents for all degrees of present bias is of size at most polynomial  in  the number of nodes. They also showed how to improve the behavior of present-biased agents by deleting edges and nodes, and they provided a characterization of the subgraphs supporting efficient agent's behavior. Finally, they analyzed the case of a collection of agents with different degrees of present bias, and showed how to divide the global task to be performed by the agents into ``easier'' sub-tasks, so that each agent performs efficiently her sub-tasks.   

As far as we are aware of, all contributions subsequent to~\cite{KleinbergO14}, and related to our paper, essentially remain within the same graph theoretic framework as~\cite{KleinbergO14}, and focus on algorithmic problems related to this framework. In particular, Albers and Kraft~\cite{Albers2018} studied the ability to place rewards at nodes for motivating and guiding the agent. They show hardness and inaproximability results, and provide an approximation algorithm whose performances match the inaproximability bound. The same authors considered another approach in~\cite{AlbersK17} for overcoming these hardness issues, by allowing not to remove edges but to increase their weight. They were able to design a 2-approximation algorithm in this context. Tang et al.~\cite{tang2017computational} also proved hardness results related to the placement of rewards, and showed that  finding a motivating subgraph is NP-hard. Gravin et al.~\cite{GravinILP16} (see~\cite{GravinILP16arch} for the full paper) extended the model by considering the case where the degree of present bias may vary over time, drawn independently at each step from a fixed distribution. In particular, they described the structure of the worst-case graph for any distribution, and derived conditions on this distribution under which the worst-case cost ratio is exponential or constant. 

Kleinberg, Oren, and Raghavan~\cite{KleinbergOR16,KleinbergOR17} revisited the model in~\cite{KleinbergO14}. In~\cite{KleinbergOR16}, they were considering agents estimating erroneously the degree $\beta$ of present bias, either underestimating or overestimating that degree, and compared the behavior of such agents with the behavior of ``sophisticated'' agents who 
are aware of their present-biased behavior in future and take this into account in their strategies. 
 In~\cite{KleinbergOR17}, they extended the model by considering not only agents suffering from present-biases, but also from \emph{sunk-cost} bias, i.e., the tendency to incorporate costs experienced in the past into one's plans for the future. Albers and Kraft~\cite{AlbersK17b} considered a model with uncertainty, bearing similarities with~\cite{KleinbergOR16}, in which the agent is solely aware that the degree of present bias belongs to some set $B \subseteq (0, 1]$, and may or may not vary over time. 

\section{Procrastination under minimization problems}\label{sec:min1}

This section includes a formal definition of inconsistent planning by present-biased agents, and describes two extreme scenarios: one in which a present-biased agent constructs worst case plannings, and one in which the plannings generated by a present-biased agent are close to optimal. 

\subsection{Model and Definition} 
\label{sec:modelanddefinition}

We consider minimization problems defined as triples  $(\m{I},F,c)$, where $\m{I}$ is the set of instances (e.g., the set of all graphs), $F$ is a function that returns the set $F(I)$ of feasible solutions for every instance $I\in\m{I}$ (e.g., the set of all edge-cuts of any given graph), and $c$ is a non-negative 
function returning the cost $c(I,S)$ of every feasible solution $S\in F(I)$ of every instance $I\in\m{I}$ (e.g., the number of edges in a cut). We focus solely on optimization problems for which 
\begin{itemize}
\item[(i)] a finite ground set $\m{S}_I\neq \varnothing$ is associated to every instance $I$, 
\item[(ii)] every feasible solution for $I$ is a set $S \subseteq \m{S}_I$, and
\item[(iii)] $c(I,S)=\sum_{x \in S}\w(x)$ where $\w:\m{S}_I\to \mathbb{N}$ is a weight function. 
\end{itemize}
Moreover, we enforce two properties that are satisfied by classical minimization problems.  Specifically we assume that: 
\begin{itemize}
\item All considered problems are closed downward, that is, for every considered minimization problem $(\m{I},F,c)$, every $I\in\m{I}$, and every $x\in \m{S}_I$, the instance $I\smallsetminus \{x\}$ defined by the feasible solutions $S\smallsetminus\{x\}$, for every $S\in F(I)$, is in~$\m{I}$ with the same weight function~$\w$ as for~$I$. This guarantees that an agent cannot be stuck after having performed some task~$x$, as the sub-problem $I\smallsetminus \{x\}$ remains solvable for every~$x$. 
\item All considered feasible solutions are closed upward, that is, for every minimization problem $(\m{I},F,c)$, and every $I\in\m{I}$, $\m{S}_I$ is a feasible solution, and, for every $S\in F(I)$, if $S \subseteq S' \subseteq \m{S}_I$ then $S'\in F(I)$. This guarantees that an agent performing a sequence of tasks $x_0,x_1,\dots$ eventually computes a feasible solution. 
\end{itemize}
Inconsistent planning can be rephrased in this framework as follows. 

\paragraph{\bf Inconsistent Planning.} 

Let $\beta<1$ be a positive constant. Given a minimization problems $(\m{I},F,c)$, the biased cost $c_\beta$ satisfies $c_\beta(S)=\w(x)+\beta \, c(S\smallsetminus \{x\})$ for every feasible solution $S$ of every instance $I\in \m{I}$, where  $x=\argmin_{y\in S}\w(y)$. Given an instance $I$, the agent aims at finding a feasible solution $S\in I$ by applying a present-based planning defined inductively as follows. Let $I_0=I$. For every $k\geq 0$, given the instance~$I_k$, the agent computes a feasible solution~$S_k$ with minimum cost $c_\beta(S_k)$ among all feasible solutions for~$I_k$. Let $x_k = \argmin_{y\in S_k}\w(y)$. The agent stops whenever  $\{x_0,x_1,\dots,x_k\}$ is a feasible solution for~$I$. Otherwise,  it carries on the construction of the solution by considering $I_{k+1}=I_k\smallsetminus \{x_k\}$. 

\medskip

Observe that inconsistent planning terminates. Indeed, since the instances of the considered problem  $(\m{I},F,c)$ are closed downwards, $I_k=I\smallsetminus \{x_0,\dots,x_{k-1}\}\in \m{I}$ for every $k\geq 0$, i.e., inconsistent planning is well defined. Moreover, since the feasible solutions are closed upward, there exists $k\geq 0$ such that $\{x_0,x_1,\dots,x_k\}$ is a feasible solution for~$I$.

The \emph{cost} of inconsistent planning is defined as the ratio $\varrho=\frac{c(S)}{\opt}$ where~$S=\{x_0,x_1,\dots,x_k\}$ is the solution returned by the agent, and $\opt=c(S_{\opt})$ is the cost of an optimal solution $S_{\opt}$ for the same instance of the considered minimization problem. 

\paragraph{\bf Approximated evaluation.}

It can happen that the considered minimization problem is computationally hard, say NP-hard, and the agent is  unable to compute a feasible solution $S$ of minimum cost $c_\beta(S)$ exactly. Then the agent can pick an approximate solution instead. For this situation, we modify the above strategy of the agent as follows. Assume that the agent has access to an $\alpha$-approximation algorithm $\mathcal{A}$ that, given an instance $I$, computes a feasible solution $S^*$ to the instance such that $c_\beta(S^*)\leq \alpha\min c_\beta(S)$, where minimum is taken over all feasible solution $S$ to $I$. For simplicity, we assume throughout the paper that $\alpha\geq 1$ is a constant, but our results can be generalized for the case, where $\alpha$ is a function of the input size or $\opt$. 

Again, the agent uses an inductive scheme to construct a solution. Initially, $I_0=I$. For every $k\geq 0$, given the instance~$I_k$, the agent computes a feasible solution~$S_k$ of cost at most $\alpha\min c_\beta(S)$, where  the minimum is taken over all feasible solutions $S$ of $I_k$. Then, exactly as before, the agent finds $x_k = \argmin_{y\in S_k}\w(y)$. 
If $\{x_0,x_1,\dots,x_k\}$ is a feasible solution for~$I$, then the agent stops. Otherwise,  we set $I_{k+1}=I_k\smallsetminus \{x_k\}$ and proceed. 
The \emph{$\alpha$-approximative cost} of inconsistent planning is defined as the ratio $\varrho_\alpha=\frac{c(S)}{\opt}$ where~$S=\{x_0,x_1,\dots,x_k\}$. Clearly, the $1$-approximative cost coincides with $\varrho$.

\subsection{Worst-Case Present-Biased Planning}\label{subsec:worst_case}

We start with a simple observation. Given a feasible solution $S$ for an instance $I$ of a minimization problem, we say that $x\in S$ is \emph{superfluous} in $S$ if $S\smallsetminus \{x\}$ is also feasible for~$I$. The ability for the agent to make superfluous choices yields trivial scenarios in which the cost ratio~$\varrho$ can be arbitrarily large. This is for instance the case of an instance of set-cover, defined as one set~$y=\{1,\dots,n\}$ of weight~$c>1$ covering all elements, and $n$ sets $x_i=\{i\}$, each of weight~1, for $i=1,\dots,n$. Every solution $S_i=\{x_i,y\}$ is feasible, for $i=1,\dots,n$, and satisfies $c_\beta(S_i)=1+\beta c$. As a result, whenever $1+\beta c<c$, the present-biased agent constructs the solution $S=\{x_1,\dots,x_n\}$, which yields a cost ratio $\varrho=n/c$, which can be made arbitrarily large as $n$ grows. Instead, if the agent is bounded to avoid superfluous choices, that is, to systematically choose \emph{minimal} feasible solutions, then only the feasible solutions $\{y\}$ and $\{x_1,\dots,x_n\}$ can be considered. As a result, 
the agent will compute the optimal solution $S_\opt=\{y\}$ if $c<1+\beta(n-1)$. 

Unfortunately, bounding the agent to systematically choose \emph{minimal} feasible solutions, i.e., solutions with no superfluous elements, is not sufficient to avoid procrastination. That is, it does not prevent the agent from computing solution with high cost ratio. This is for instance the case of another instance of set-cover, that we denote by~$I_{SC}^{(n)}$ for further references. 

\paragraph{Set-cover instance $I_{SC}^{(n)}$:} specified by $2n$ subsets of $\{1,\dots,n\}$ defined as $x_i=\{i\}$ with weights~$1$, and $y_i=\{i,\dots,n\}$ with weight~$c>1$, for $i=1,\dots,n$.  

\medskip

The minimal feasible solutions of $I_{SC}^{(n)}$ are $\{y_1\}$ of weight~$c$, $\{x_1,\dots,x_i,y_{i+1}\}$ of weight $i+c$ for $i=1,\dots,n-1$, and $\{x_1,\dots,x_n\}$ of weight~$n$. Whenever $1+\beta c<c$, a time-biased agent bounded to make non-superfluous choices only yet constructs the solution $\{x_1,\dots,x_n\}$ which yields a cost ratio $\varrho=n/c$, which can be made arbitrarily large as $n$ grows. We need the following lemma about biased solutions for minimization problems.

\begin{lemma}\label{lem:biased}
Let $\alpha\geq 1$ and let $S^*$ be a feasible solution for minimization problem, satisfying $c_\beta(S^*)\leq\alpha \min c_\beta(S)$, where the minimum is taken over all feasible solutions. Then 
\begin{itemize}
\item[(i)] $\w(x)\leq \alpha \cdot \opt$ for $x=\argmin_{y\in S^*}\w(y)$, and 
\item[(ii)] $c(S^*)\leq \frac{\alpha}{\beta}\opt$.
\end{itemize}
\end{lemma}

\begin{proof}
Let $S$ be an optimum solution.  As $\beta<1$, it follows that
$
\w(x)\leq \w(x)+\beta \cdot\w(S^*\setminus \{x\})
= c_\beta(S^*)
\leq \alpha \cdot c_\beta(S)
\leq \alpha \cdot c(S)
=\alpha \cdot\opt,
$
and this proves (i). To show (ii), note that
$
c(S^*) = \w(x)+\w(S^*\setminus\{x\})=\frac{1}{\beta}(\beta \w(x)+\beta\w(S^*\setminus\{x\})),
$
from which it follows that 
$
c(S^*) \leq  \frac{1}{\beta}(\w(x)+\beta\w(S^*\setminus \{x\}))=
\frac{1}{\beta}c_\beta(S^*)\leq \frac{\alpha}{\beta} c_\beta(S)\leq \frac{\alpha}{\beta}c(S)=\frac{\alpha}{\beta}\opt,
$
which completes the proof. 
\end{proof}

Lemma~\ref{lem:biased} has a simple consequence that also can be derived from the results of Gravin et al.~\cite[Claim~5.1]{GravinILP16arch}, that we state as a theorem despite its simplicity, as it illustrates one major difference between our model and the model in~\cite{KleinbergO18}. 

\begin{theorem}\label{theo:linearbound} 
For every $\alpha\geq 1$ and every minimization problem, the $\alpha$-appro\-xi\-mative cost ratio~$\varrho_\alpha$ cannot exceed $\alpha \cdot k$ where $k$ is the number of steps performed by the agents to construct the feasible solution $\{x_1,\dots,x_k\}$ by following the time-biased strategy. 
\end{theorem}

\begin{proof}
By Lemma~\ref{lem:biased}(i), at any step~$i\geq 1$ of the construction, the agent adds an element $x_i\in\m{S}_I$ in the current partial solution, and this element satisfies 
$
\w(x_i)\leq \alpha \, c_\beta(S_\opt) \leq \alpha \, c(S_\opt) =\alpha \cdot \opt. 
$
Therefore, if the agent computes a solution $\{x_1,\dots,x_k\}$, then the $\alpha$-approximative cost ratio for this solution satisfies
$
\varrho_\alpha = \sum_{i=1}^k\w(x_i)/\opt \leq \alpha\, k,  
$
as claimed.
\end{proof}

\paragraph{Remark.} 
The bound in Theorem~\ref{theo:linearbound} is in contrast to the general model in~\cite{KleinbergO18}, in which an agent performing $k$ steps can incur a cost ratio exponential in~$k$. This is because the model in~\cite{KleinbergO18} enables to construct graphs with arbitrary weights. In particular, in a graph such as the one depicted on Fig.~\ref{fig:Akerlof}, one can set up weights such that the weight of $(v_1,t)$ is a constant time larger than the weight of $(s,t)$, the weight of $(v_2,t)$ is  in turn a constant time larger than the weight of $(v_1,t)$, etc., and still a present-biased agent starting from~$s$ would travel via $v_1,v_2,\dots,v_k$ before reaching~$t$.  In this way, the sum of the weights of the edges traversed by the agent may become exponential in the number of traversed edges. This phenomenon does not occur when focussing on minimization tasks. Indeed, given a partial solution, the cost of completing this solution into a global feasible solution cannot exceed the cost of constructing a global feasible solution from scratch. 

\medskip

It follows from  Theorem~\ref{theo:linearbound} that $I_{SC}^{(n)}$ is a worst-case instance. Interestingly, this instance fits with realistic procrastination scenarios in which the agent has to perform a task (e.g., learning a scientific topic~$T$) by either energetically embracing the task (e.g., by reading a single thick book on topic~$T$), or starting first by an easier subtask (e.g., by first reading a digest of a subtopic of topic~$T$), with the objective of working harder later, but underestimating the cost of this postponed hard work.  The latter strategy may result in procrastination, by performing a very long sequence of subtasks $x_1,x_2,\dots,x_n$. 

\medskip

In  fact, $I_{SC}^{(n)}$ appears to be \emph{the essence of procrastination} in the framework of minimization problems. Indeed, we show that if the cost ratio is large, then the considered instance~$I$ contains an instance of the form $I_{SC}^{(n)}$ with large~$n$. More precisely, we say that an instance $I$ contains an instance $J$ as a \emph{minor} if the ground set $\m{S}_J$ associated to~$J$ is a collection of subsets of the ground set $\m{S}_I$ associated to~$I$, that is 
$
\m{S}_J\subseteq 2^{\m{S}_I},
$
and, for every $\bar{S}\subseteq \m{S}_J$, 
$
 \mbox{$\bar{S}$ is feasible for $J$} 
 $
 if and only if 
 $
  S=\bigcup_{\bar{x}\in\bar{S}}\bar{x} \; \; \mbox{is feasible for $I$.} 
$
Moreover, the weight function $\bar{\w}$ for the elements of $\m{S}_J$ must be induced by the one for $\m{S}_I$ as
$
\bar{\w}(\bar{x})=\sum_{x\in\bar{x}}\w(x)
$
for every $\bar{x}\in\m{S}_J$. Let $J^{(n)}$ be any instance of a minimization problem such that its associated ground set is 
$
\m{S}_{J^{(n)}}=\{x_1,\dots,x_n\}\cup\{y_1,\dots,y_n\},
$
and the set of feasible solutions for $J^{(n)}$ is
$
F(J^{(n)})=\big\{ \{y_1\}, \{x_1,y_2\},\{x_1,x_2,y_3\},\dots,\{x_1,\dots,x_{n-1},y_n\},\\  \{x_1,\dots,x_n\} \big\}.
$
The following result sheds some light on \emph{why} the procrastination structure of Fig.~\ref{fig:Akerlof} pops up.

\begin{theorem}\label{theo:minor}
Let $I$ be an instance of a minimization problem for which the present-biased agent with parameter~$\beta\in(0,1)$ computes a solution for~$I$ with cost $\alpha \cdot \opt(I)$ for some $\alpha>1$. Then $I$ contains~$J^{(n)}$ as a minor for some $n\geq \alpha$, and  the present-biased agent with parameter~$\beta$ computes a solution for~$J^{(n)}$  with cost $\alpha\cdot \opt(J^{(n)})$.
\end{theorem}

\begin{proof}
Let $S=\{x_1,\dots,x_n\}$ be the final solution selected by the present-biased agent for~$I$, and let~$\w$ be the weight function on the set $\m{S}_I$ associated to~$I$. We have $\sum_{i=1}^n\w(x_i) = \alpha\,\opt(I)$. For every $i\in\{1,\dots,n\}$, let us denote by $\opt(I\smallsetminus\{x_1,\dots,x_i\})$ the cost of an optimal solution for the instance $I\smallsetminus\{x_1,\dots,x_i\}$, and by $S_{ \opt(I\smallsetminus\{x_1,\dots,x_i\})}$ a corresponding optimal solution. For $i=0$, $S_{ \opt(I\smallsetminus\{x_1,\dots,x_i\})}$ is an optimal solution for $I$. For $i=1,\dots,n$, we define 
$$
\bar{x}_i = \{x_i\}, 
\;
\mbox{and}
\;
\bar{y}_i=S_{ \opt(I\smallsetminus\{x_1,\dots,x_{i-1}\})}.
$$
Let $J$  be the instance with ground set 
$
\{\bar{x}_1,\dots,\bar{x}_n\}\cup\{\bar{y}_1,\dots,\bar{y}_n\},
$
and feasible solutions 
$$
\{\bar{y}_1\}, \{\bar{x}_1,\bar{y}_2\},\dots,\{\bar{x}_1,\dots,\bar{x}_{n-1},\bar{y}_n\}, \{\bar{x}_1,\dots,\bar{x}_n\}.
$$
Note that $\bar{x}_i \neq \bar{x}_j$ for every $i\neq j$, because $x_i\neq x_j$  for every $i\neq j$. Also, for every $i\in\{1,\dots,n-1\}$ and $j\in\{1,\dots,n\}$, $\bar{x}_i \neq \bar{y}_j$, because otherwise the sequence constructed by the time-biased agent for~$I$ would stop at~$x_k$  with $k<n$. Therefore, we have $J=J^{(n)}$, and, since $\w(x_i)\leq \opt(I)$ for every $i=1,\dots,n$, $n\geq \alpha$ holds. 

For analyzing the behavior of a present-biased agent with parameter~$\beta$ acting on~$J$, let us assume that $k$ steps were already performed by the agent, with $0\leq k < n$, resulting in constructing the partial solution $\{\bar{x}_1,\dots,\bar{x}_k\}$. (For $k=0$, this partial solution is empty). The feasible solutions for $J_k=J\smallsetminus \{\bar{x}_1,\dots,\bar{x}_k\}$ are 
\[
\{\bar{y}_1\},\dots, \{\bar{y}_{k+1}\}, \{\bar{x}_{k+1},\bar{y}_{k+2}\},\dots,\{\bar{x}_{k+1},\dots,\bar{x}_{n-1},\bar{y}_n\}, \{\bar{x}_{k+1},\dots,\bar{x}_n\}.
\]
Note that, for every $i\in\{1,\dots,n\}$, $\bar{\w}(\bar{x}_i)=\w(x_i)$, and  $\bar{\w}(\bar{y}_i)=\opt(I\smallsetminus\{x_1,\dots,x_{i-1}\})$. We claim that $\bar{x}_{k+1}$ is the next element chosen by the agent. Indeed, note first that $\bar{\w}(\bar{x}_{k+1})\leq \bar{\w}(\bar{y}_{k+2})$, as, otherwise, we would get $\w(x_{k+1}) > \opt(I\smallsetminus\{x_1,\dots,x_{k+1}\})$, contradicting the choice of $x_{k+1}$ by the agent performing on~$I$. As a consequence, 
$$
c_\beta(\{\bar{x}_{k+1},\bar{y}_{k+2}\})=\bar{\w}(\bar{x}_{k+1}) + \beta \,\bar{\w}(\bar{y}_{k+2}). 
$$
It follows from the above that, for every $j=1,\dots,k+1$, $c_\beta(\{\bar{y}_j\})\geq c_\beta(\{\bar{x}_{k+1},\bar{y}_{k+2}\})$, as the reverse inequality would contradict the choice of $x_{k+1}$ by the agent performing on~$I$. For the same reason, for every $\ell\in\{k+1,\dots,n-1\}$, and every $i\in\{k+1,\dots,\ell\}$, we have 
\[
c_\beta(\{\bar{x}_{k+1},\bar{y}_{k+2}\}) \leq \bar{\w}(\bar{x}_i)+\beta \Big (\sum_{j\in\{k+1,\dots,\ell\}\smallsetminus \{i\}} \bar{\w}(\bar{x}_j) +\bar{\w}(\bar{y}_{\ell+1})\Big)
\]
and
\[
c_\beta(\{\bar{x}_{k+1},\bar{y}_{k+2}\}) \leq \bar{\w}(\bar{y}_{\ell+1})+\beta \sum_{j\in\{k+1,\dots,\ell\}} \bar{\w}(\bar{x}_j).
\]
As a consequence, the present-biased agent performing on $J$ picks $\bar{x}_{k+1}$ at step~$k+1$, as claimed. The cost of the solution computed by the agent is $\sum_{i=1}^n\bar{\w}(\bar{x}_i)=\sum_{i=1}^n w(x_i)=\alpha\,\opt(I)$. On the other hand, by construction, $\opt(J)=\bar{\w}(\bar{y}_1)=\opt(I)$. The cost ratio of the solution computed by the agent for~$J$ is thus~$\alpha$, which completes the proof. 
\end{proof}

\subsection{Quasi-Optimal Present-Biased Planning}

In the previous section, we have observed that forcing the agent to avoid superfluous choices, by picking minimal feasible solutions only, does not prevent it from constructing solutions that are arbitrarily far from the optimal. In this section, we show that, by enforcing consistency in the sequence of partial solutions constructed by the agent, such bad behavior does not occur. More specifically, given a feasible solution $S$ for $I$, we say that $x$ is \emph{inconsistent} with $S$ if $x\notin S$. The following result shows that inconsistent choices is what causes high cost ratio. 

\begin{theorem} \label{theo:consistent-is-almost-opt}
An agent using an $\alpha$-approximation algorithm bounded to avoid  inconsistent choices with respect to the feasible solutions used in the past for constructing the current partial solution returns an $\alpha/\beta$-approximation of the optimal solution. This holds independently from whether the agent makes superfluous choices or not. 
\end{theorem} 

\begin{proof}
Let $I$ be an instance of a minimization problem $(\m{I},F,c)$. Let $S=\{x_0,\dots,x_k\}$ be the solution constructed by the agent for~$I$, where $x_i$ is the element computed by the agent at step~$i$, for $i=0,\dots,k$. Let $S_i$ be the feasible solution of $I_i=I\smallsetminus\{x_0,\dots,x_{i-1}\}$ considered by the agent at step~$i$. Since the agent is bounded to avoid any inconsistent choices with respect to the past, we have $x_i \in \cap_{j=0}^i S_j$ for every $i=0,\dots,k$ because $x_i\notin S_j$ for some $j<i$ would be an inconsistent choice. It follows that $S\subseteq S_0$. Therefore, $c(S)\leq c(S_0)$. Since the agent uses an  $\alpha$-approximation algorithm, by Lemma~\ref{lem:biased}(ii), $c(S_0)\leq\frac{\alpha}{\beta}\opt$ and the claim follows. 
\end{proof}

\section{Min/maximization with under/overestimation}\label{sec:min2}

We first investigate the cost ratio for minimization problems for the case when $\beta>1$.
Similar bound was obtained by Kleinberg et al. (see~\cite[Theorem~2.1]{KleinbergOR16}). However, their theorem is about sophisticated agents and  cannot be applied in our case directly.

\begin{theorem}\label{thm:min-betamoreone} Solutions computed by present-biased agents satisfy the following: 
 For
 any minimization problem with $\beta>1$, the 
 cost ratio is at most~$\beta$. 
\end{theorem}

\begin{proof}
For the proof of the theorem it is  convenient to switch to the original graph-theoretic model of Kleinberg and Oren~\cite{KleinbergO18}. Note that the task graphs corresponding to optimization problems are, in fact, directed acyclic graphs. Hence, we only consider task graphs of this type to avoid dealing with  paths of maximum length in the presence of cycles. 

Let $G$ be a directed acyclic graph (DAG) with a source $s$. Let also $\w\colon E(G)\rightarrow \mathbb{N}$ be a weight function. 
The aim of the agent is to go from the source $s$ to a sink $t$ of $G$ making present-biased decisions on each step. We assume that $G$ has an $s$-$t$ path. Let $\beta$ be a positive constant distinct from~1. 
Let $c^{\min}(x)$  be the minimum   length of an $x$-$t$ path.  

We suppose that the agent is equipped with an algorithm $\mathcal{A}$ that, given a vertex $v\in V(G)$, finds a vertex $x^*\in N_G^+(v)$ such that  
\[\w(vx^*)+\beta  \cdot  c^{\min}(x^*)=\min_{x\in N_G^+(v)}(\w(vx)+\beta  \cdot  c^{\min}(x)).\]
The agent constructs an $s$-$t$ path as follows:
if the agent occupies a vertex $v\neq t$, then he makes the  present-biased $\alpha$-approximate estimation  of the length of a shortest 
$v$-$t$ path  and moves to $x^*$. 
Note that since $G$ is a DAG, the agent would eventually  arrive to $t$. 

We denote by $\cost^{\min}(v)$  the length of a $v$-$t$ path constructed by the agent from $v$. Notice that this value is not uniquely defined as the agent may be able to choose distinct vertices that provide  $\alpha$-approximate present-biased evaluations but could give distinct lengths for the constructed paths.  
Then the proof of the theorem is implied by the following claim. 

\begin{claim-n}\label{lem:ratio-maxmin}
Let $G$ be a weighted DAG with a weigh function $\w\colon E(G)\rightarrow \mathbb{N}$ and a sink $t$. Then for every $v\in V(G)$,
  if $\beta>1$, then $\cost^{\min}(v)\leq \beta \cdot c^{\min}(v)$.
\end{claim-n}

The claim is trivial if $v=t$. Assume that $v\neq t$, and that the claim holds for every out-neighbor $x$ of $v$. Assume that 
$x^*\in N_G^+(v)$ is computed by $\mathcal{A}$, and let 
$$
y=\argmin_{x\in N_G^+(v)}(\w(vx)+c^{\min}(x)).
$$
 That is, there is a shortest $v$-$t$ path that goes through $y$.
 
By induction, we have that 
$$
\cost^{\min}(x^*)\leq\beta \cdot c^{\min}(x^*). 
$$
It follows that  
\begin{align*}
\cost^{\min}(v)=& \; \w(vx^*)+\cost^{\min}(x^*) 
\leq \;   \w(vx^*)+\beta \cdot c^{\min}(x^*)\\
\leq & \w(vx^*)+\beta \cdot c^{\min}(x^*) 
=   \min_{x\in N_G^+(v)}(\w(vx)+\beta  \cdot  c^{\min}(x))\\
\leq & \w(vy)+\beta  \cdot  c^{\min}(y)
\leq \; \beta \cdot\w(vy)+\beta  \cdot  c^{\min}(y)
\leq  \; \beta  \cdot  c^{\min}(v) .
\end{align*}
The last inequality completes the proof of the claim, and of the theorem. 
\end{proof}

Next, we  consider maximization problems. The formalism for these variants can be set up in a straightforward manner by adapting the framework displayed in Section~\ref{sec:modelanddefinition}. We establish the following worst-case bounds.  
 
\begin{theorem}\label{thm:min-max-beta} Solutions computed by present-biased agents satisfy the following: 
\begin{itemize}
\item[(i)] For 
any maximization problem with $\beta<1$, the cost ratio is at most~$\frac{1}{\beta}$;
 \item[(ii)]
 For 
 any maximization problem with $\beta>1$, the  cost ratio is at most~$\beta^c$, where $c\leq\opt$ is the cost of a solution constructed by the agent. 
\end{itemize}
\end{theorem}

\begin{proof}
As in the proof of Theorem~4, we 
  switch to the  graph-theoretic model of Kleinberg and Oren~\cite{KleinbergO18}. 
Let $G$ be a directed acyclic graph (DAG) with a source $s$ and weight function  $\w\colon E(G)\rightarrow \mathbb{N}$. 
The agent aims to go from the source $s$ to a sink $t$ of $G$ making present-biased decisions on each step. We assume that $G$ has an $s$-$t$ path. Let $\beta$ be a positive constant distinct from~1, and let 
    $c^{\max}(x)$ be the  maximum length of an $x$-$t$ path.  

Let $\alpha\in(0,1]$.  As in Theorem~4, we assume that the agent is equipped with an algorithm $\mathcal{A}$ that, given a vertex $v\in V(G)$, finds a vertex $x^*\in N_G^+(v)$ such that  
\[ \w(vx^*)+\beta\cdot c^{\max}(x^*)= \max_{x\in N_G^+(v)}(\w(vx)+\beta\,c^{\max}(x)).\] 
Using $\mathcal{A}$, the agent located in vertex  $v\neq t$
constructs an $s$-$t$ path as follows:
  the agent computes the  present-biased $\alpha$-approximate estimation  of the length of a   longest $v$-$t$ and moves to $x^*$. 
 We denote by   $\cost^{\max}(v)$ the length of a $v$-$t$ path constructed by the agent from $v$.

\begin{claim-n}\label{lem:ratio-maxmin2a}
Let $G$ be a weighted DAG with a weigh function $\w\colon E(G)\rightarrow \mathbb{N}$ and a sink $t$. Then for every $v\in V(G)$,
\begin{itemize}
\item[\rm (i)] if $\beta<1$, then $\cost^{\max}(v)\geq \beta\, c^{\max}(v) $,
\item[\rm (ii)] if $\beta>1$, then $c^{\max}(v)\leq  \cost^{\max}(v)\,{\beta}^{\cost^{\max}(v)}$.
\end{itemize}
\end{claim-n}
 As in Theorem~4, we prove the claim by induction. 
The claim is trivial if $v=t$. Assume that $v\neq t$, and that the claim holds for every out-neighbor $x$ of $v$. 
Let
$x^*$ be the vertex computed by $\mathcal{A}$ and let 
$$
y=\argmax_{x\in N_G^+(v)}(\w(vx)+c^{\max}(x)). 
$$
That is, there is a longest $v$-$t$ path that goes through $y$.

To show (i), we use the inductive assumption that 
$$\cost^{\max}(x^*)\geq  \beta\, c^{\max}(x^*).$$
We have 
\begin{align*}
\cost^{\max}(v)=& \; \w(vx^*)+\cost^{\max}(x^*) 
\geq \; \w(vx^*)+\beta\, c^{\max}(x^*)\\
= & \max_{x\in N_G^+(v)}(\w(vx)+\beta\,c^{\max}(x)) 
\geq \; \w(vy)+\beta\,c^{\max}(y)\\
\geq \; & \beta\,\w(vy)+\beta\,c^{\max}(y) 
\geq \; \beta\,c^{\max}(v).
\end{align*}

To prove (ii), we assume that 
the following inductive assumption holds: 
$$
c^{\max}(x^*)\leq  \cost^{\max}(x^*)\,\beta^{\cost^{\max}(x^*)}.
$$
It follows that
\begin{align*}
c^{\max}(v)=&\; \w(vy)+c^{\max}(y) 
\leq \; \w(vy)+\beta\cdot c^{\max}(y)\\
\leq &\; \max_{x\in N_G(v)}(\w(vx)+\beta\cdot c^{\max}(x))
\leq \;  \w(vx^*)+\beta\cdot c^{\max}(x^*) \\
\leq &\;  \w(vx^*)+\beta\cdot \cost^{\max}(x^*)\,\beta^{\cost^{\max}(x^*)}\\ 
\leq & \; \beta^{\cost^{\max}(x^*)+1}(\w(vx^*)+\cost^{\max}(x^*)) 
 =    \; \cost^{\max}(v)\,\beta^{\cost^{\max}(x^*)+1}\\
 \leq &\; \cost^{\max}(v)\,\beta^{\cost^{\max}(v)}.
\end{align*}
This last inequality completes the proof of Claim~\ref{lem:ratio-maxmin2a}, which immediately gives the bounds for the $\alpha$-approximate cost ratio claimed in the statement of the theorem.
\end{proof}

We also can write the bound for the cost ratio for $\beta>1$ in the following form to obtain the upper bound that depends only on the value of $\opt$.

\begin{corollary}\label{cor:upper} 
For any maximization problem with $\beta>1$, the   cost ratio is at most $(1+\log\beta)\frac{\opt}{\log\opt}$.
\end{corollary} 

\begin{proof}
Let $c$ be the cost of a solution constructed by the agent. By Theorem~\ref{thm:min-max-beta}, $\opt\leq c\beta^c$. Therefore, 
$\log\opt\leq \log c+c\log\beta \leq \big(1+\log\beta)c,$
and
$\frac{\opt}{c}\leq (1+\log\beta)\frac{\opt}{\log\opt}.$
\end{proof}

For minimization problems with $\beta>1$, and maximization problems with $\beta<1$, we have that the cost ratio is bounded by a constant. This differs drastically with the case of maximization problems with $\beta>1$, when the cost ratio is still bounded but the bound is exponential. This exponential upper bound  is however essentially tight, in the sense that the exponent cannot be avoided.

\begin{theorem}\label{theo:lwbound} 
There are maximization problems for which a present-biased agent with $\beta>1$ returns a solution whose  cost ratio is at least~$\frac{1}{c}\,\beta^{c-1}$, where $c$ is the cost of the solution constructed by the agent. 
\end{theorem}

\begin{proof}
Let us consider the maximum independent set problem. In this problem, we are given a weighted graph $G$, and the task is to find an independent set of maximum weight. Let $k$ be a positive integer. We construct the graph $G_k$ as follows (see Fig.~\ref{fig:Gk}):
\begin{itemize}
\item construct $k+1$ vertices $x_0,\ldots,x_k$, and make them pairwise adjacent,
\item construct $k$ vertices $y_1,\ldots,y_k$,
\item for each $i\in\{1,\ldots,k\}$, make $y_i$ adjacent to $x_i,x_{i+1},\ldots,x_k$.
\end{itemize}
To define the weights, let $\beta\geq 2$. We set $\w(y_i)=1$ for every $i\in\{1,\ldots,k\}$, and
$\w(x_i)=\beta^i$ for every $i\in\{0,\ldots,k\}$.

\begin{figure}[h]
\begin{center}
\scalebox{0.8}{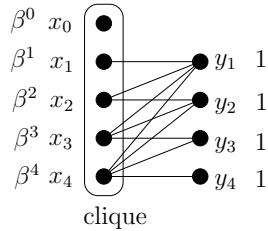}
\end{center}
\vspace*{-2ex}
\caption{Construction of $G_k$ for $k=4$.}\label{fig:Gk}
\end{figure}

Since $X=\{x_0,\ldots,x_k\}$ is a clique,  any independent set has at most one vertex in $X$. Therefore, the family of sets 
$$
S_i=\{x_i\}\cup\{y_{i+1},\ldots,y_{k}\}
$$ 
for $i\in\{0,\ldots,k\}$ form the family of maximal independent sets. Because $\beta\geq 2$, it is straightforward to verify that the single-vertex set $S_k=\{x_k\}$ is an independent set of maximum weight  $\beta^k$, that is, $\opt=\beta^k$. Observe that the biased cost of this set is $\beta^k$ as well. 

From the other side, the biased cost of $S_{k-1}$ is 
$$
\w(y_k)+\beta\cdot \w(x_{k-1})=1+\beta\cdot \beta^{k-1}=1+\beta^k>\beta^k.
$$
Hence, the agent would prefer to select $y_k$ at the first iteration. At the next iteration, the agent considers the graph obtained from $G_k$ by the deletion of $y_k$ and its neighborhood, that is, $G_{k-1}$. Applying the same arguments inductively, we conclude that the agent will end up with the set $S_0=\{x_0,y_1,\ldots,y_k\}$ with $\w(S_0)=k+1$.
We obtain that $\opt=\beta^{c-1}$, where $c$ is the  cost of a solution constructed by the agent. \qed
\end{proof}

\paragraph{Remark.} An example similar to the one in the proof of Theorem~\ref{theo:lwbound} can be constructed for the knapsack problem. Recall that in this problem, we are given $n$ objects with positive integer \emph{values}~$v_i$, and \emph{weights}~$w_i$, for $i\in\{1,\ldots,n\}$, and $W\in\mathbb{N}$. the task is to find a set of objects $S\subseteq\{1,\ldots,n\}$ of maximum value with the total weight at most $W$. Let $k$ be a positive integer. We let $n=2k+1$, $W=n$ and $\beta\geq 2$. We define 
$$
v_i=\beta^{k+1-i}, 
$$
and 
$$
w_i=W-(i-1)
$$ 
for every $i\in\{1,\ldots,k+1\}$, and we set $v_i=1$ and $w_i=1$ for $i\in\{k+2,\ldots,n\}$. Using the same arguments as for the maximum independent set problem, we obtain that the optimum solution has cost $\beta^k$ while a present-biased agent would select a solution of cost $k+1$. 

\section{Covering and hitting problems}\label{sec:hitting}

In Section~\ref{subsec:worst_case}, we have seen several instances of the set-cover problem whose cost ratio cannot be bounded by any  function of \opt. The same obviously holds for the hitting-set problem. Recall that an instance  of hitting-set is defined by a collection $\Sigma$ of subsets of a finite set $V$, and the objective is to find the subset $S \subseteq V$ of minimum size, or minimum weight,  which intersects (hits) every set in~$\Sigma$. However, set-cover problems, and hitting set problems behave differently when the sizes of the sets are bounded. First, we consider the $d$-set cover problem. 

\paragraph{\bf The $d$-set cover problem.}
Let $d$ be a positive integer. The task of the  $d$-set cover problem is, given a collection $\Sigma$ of subsets with size at most~$d$ of a finite set $V$, and given a weight function $\w\colon \Sigma\rightarrow \mathbb{N}$, find a set $S\subseteq \Sigma$ of minimum weight that covers $V$, that is, $\bigcup_{X\in S}X=V$. 

\begin{theorem}\label{d-cover}
Let $\alpha\geq 1$.
For any instance of the $d$-set-cover problem, the $\alpha$-approximative cost ratio is at most $\alpha \cdot d\cdot \opt$. 
\end{theorem}

\begin{proof}
Let $I=(\Sigma,V,\w)$ be an instance of the $d$-set cover problem. 
Let $|V|=n$. Denote by $S_1,\ldots,S_p$ a sequence of solutions computed by the present-biased agent avoiding superfluous solutions, and let $S=\{X_1,\ldots,X_p\}$ be the obtained solution for  $I$. That is,  $S$ is a set cover such that  
$$
X_i=\argmin_{Y\in S_i}\w(Y)
$$ 
for $i\in\{1,\ldots,p\}$.  Clearly, $p\leq n$. Note that for each iteration $i\in\{1,\ldots,p\}$, the agent considers the instance 
$I_i=(\Sigma_i,V_i,\w_i)$, where, as superfluous solutions are avoided,
\[
\begin{array}{lcl}
V_i & = & V\setminus (\bigcup_{j=1}^{i-1}X_i), \\
\Sigma_i& =& \{X\cap V_i\mid X\in \Sigma\setminus\{X_1,\ldots,X_{i-1}\}\},\\
\w_i& =& \w|_{\Sigma_i}.
\end{array}
\]
We have that, for every $i\in\{1,\ldots,p\}$, 
$$
\opt=\opt(I)=\opt(I_1)\geq\cdots\geq\opt(I_p), 
$$
and, by Lemma~1(ii), 
$$
\w(X_i)\leq\alpha\,\opt(I_i). 
$$ 
Since each set of $\Sigma$ covers at most $d$ elements of $V$, $\opt(I)\geq n/d$. Therefore
$$
c(S)=\sum_{i=1}^p\w(X_i)\leq \alpha\, \opt(I)n\leq \alpha d\,\opt(I)^2.
$$
It follows that the cost ratio is at most $\alpha d\,\opt$.
\end{proof}

\paragraph{\bf The $d$-hitting set problem. } Let $d$ be a positive integer. We are given a  collection $\Sigma$ of subsets with size $d$ of a finite set $V$, a weight function $\w\colon V\rightarrow \mathbb{N}$. The task is to find a set $S\subseteq V$ of minimum weight that hits every set of $\Sigma$. 

We are using the classical \emph{sunflower lemma} of Erd\H{o}s and Rado~\cite{ErdosR60}. We state this result in the form given in~\cite{CyganFKLMPPS15}. A \emph{sunflower} with $k$ \emph{petals} and  a \emph{core} $X$ is a collection of pairwise distinct sets $S_1,\ldots,S_k$ such that $S_i\cap S_j=X$ for all distinct $i,j\in\{1,\ldots,k\}$. Note that the core may be empty, that is, a collection of $k$ pairwise disjoint sets is a sunflower. 

\begin{lemma}[{Sunflower lemma~\cite{ErdosR60}}]\label{lem:sunflower}
Let $\mathcal{A}$ be a family of pairwise distinct sets over a universe $U$ such that for every $A\in\mathcal{A}$, $|A|=d$. If $|\mathcal{A}|>d!(k-1)^d$, then $\mathcal{A}$ contains a sunflower with $k$ petals. 
\end{lemma}

\begin{theorem}\label{d-hitting}
Let $\alpha\geq 1$.
For any instance of the $d$-hitting-set problem, the $\alpha$-approximative cost ratio is at most $\alpha d!\,(\frac{\alpha}{\beta}\opt)^{d}$. 
\end{theorem}

\begin{proof}
Let  $I=(\Sigma,V,\w)$ be an instance of the $d$-hitting set problem.  
Denote by $S_1,\ldots,S_p$ a sequence of solutions computed by the present-biased agent avoiding superfluous solutions, and let $S=\{v_1,\ldots,v_p\}$ be the obtained solution for  $I$, that is,  $S\subseteq V$ is a hitting set  such that  
$$
v_i=\argmin_{v\in S_i}\w(v)
$$
for every $i\in\{1,\ldots,p\}$. Since the agent avoids  superfluous solutions, the agent considers the instance $I_i=(\Sigma_i,V_i,\w_i)$ at each iteration $i\in\{1,\ldots,p\}$, where 
\[
\begin{array}{lcl}
V_i & = & V\setminus\{v_1,\ldots,v_{i-1}\},\\
\Sigma_i & = & \{X\in\Sigma\mid X\cap \{v_1,\ldots,v_{i-1}\}=\emptyset\}, \\
\w_i&=&\w|_{V_i}. 
\end{array}
\]
Let $i\in\{1,\ldots,p\}$. Since $S_i$ is a minimal hitting set for $\Sigma_i$, there is $X_i\in \Sigma_i$ such that $v_i\in X_i$, and, for every $v\in S_i\setminus\{v_i\}$, $v\notin X_i$. We say that $X_i$ is a \emph{private} set for $v_i$. Observe that 
$$
\opt=\opt(I)=\opt(I_1)\geq\cdots\geq \opt(I_p)
$$ 
by the construction of the instances. The following claim is crucial for the prof of the theorem.

\begin{claim-n}\label{cl:step-a}
$p \leq d!\,(\frac{\alpha}{\beta}\opt)^d$.
\end{claim-n}

Let us assume, for the purpose of contradiction, that $p> d!\,(\frac{\alpha}{\beta}\opt)^d$. Consider the private sets $X_1,\ldots,X_p\in\Sigma$ for $v_1,\ldots,v_p$, respectively, and let $\mathcal{A}=\{X_1,\ldots,X_p\}$. Note that $X_1,\ldots,X_p$ are pairwise distinct, since $X_h\in \Sigma_h$ and $X_h\notin \Sigma_{h+1}$ for every $h\in\{1,\ldots,p-1\}$. Let 
$$
k= \Big\lfloor \frac{\alpha}{\beta}\opt  \Big \rfloor +1\geq 2.
$$ We have that $|X_h|=d$ for every $X_h\in\mathcal{A}$, and $|\mathcal{A}|>d!\,(k-1)^d$. Hence, $\mathcal{A}$ contains a sunflower with $k$ petals by Lemma~\ref{lem:sunflower}. 
Denote by $X_{h_1},\ldots,X_{h_k}$ the sets of the sunflower, and let $Y$ be its core. 

Suppose that $Y=\emptyset$. Then, $X_{h_1},\ldots,X_{h_k}$ are pairwise disjoint. Therefore, 
$$
\opt(I)\geq k=\Big\lfloor \frac{\alpha}{\beta}\opt  \Big\rfloor+1>\opt, 
$$
which is a contradiction. Therefore, $Y\neq\emptyset$.

We show that, for every hitting set $R$ for $\Sigma_{h_1}$ of weight at most $ \frac{\alpha}{\beta}\opt$, $R\cap Y\neq\emptyset$. Assume that this is not the case, that is, $R\cap Y=\emptyset$. Since $R$ is a hitting set, there exists
$$
u_\ell\in X_{h_\ell}\setminus Y
$$
such that $u_\ell\in R$ for every $\ell\in\{1,\ldots,k\}$. Because $\{X_{h_1},\ldots,X_{h_k}\}$ is a sunflower, $u_1,\ldots,u_k$ are distinct. It follows that
$$
\w(R)\geq \sum_{\ell=1}^k\w(u_\ell)\geq k=\Big\lfloor \frac{\alpha}{\beta}\opt \Big\rfloor+1>\frac{\alpha}{\beta}\opt.$$
The latter strict inequality is contradicting the fact that the weight of $R$ is at most $ \frac{\alpha}{\beta}\opt$. We conclude that $R\cap Y\neq\emptyset$.

Recall that $S_{h_1}$ is a feasible solution for $I_{h_1}$, and, by Lemma~1 (ii), 
$$
\w(S_{h_1})\leq  \frac{\alpha}{\beta}\opt(I_{h_1})\leq  \frac{\alpha}{\beta}\opt. 
$$
Then $S_{h_1}\cap Y\neq\emptyset$.
The set $X_{h_1}$ was chosen to be a private set for $v_{h_1}$, that is, $S_{h_1}\cap X_{h_1}=\{v_{h_1}\}$. Note that $v_{h_1}\notin X_{h_2}$. Hence, $v_{h_1}\notin Y$, and thus $S_{h_1}\cap Y=\emptyset$. This is a contradiction, which completes the proof of the claim.

By  Claim~\ref{cl:step-a}, $p\leq d!\,(\frac{\alpha}{\beta}\opt)^d$. By Lemma~1 (i),
$$
\w(v_i)\leq \alpha\, \opt(I_i)\leq \alpha\,\opt
$$
for every $i\in\{1,\ldots,p\}$. Therefore, 
$$
c(S)=\sum_{i=1}^p\w(v_i)\leq \alpha\, \opt\,p\leq \alpha\,\opt\,d!\,(\frac{\alpha}{\beta}\opt)^d.
$$
It follows that $\frac{c(S)}{\opt}\leq \alpha d!\,(\frac{\alpha}{\beta}\opt)^{d}$.  

Let $S=\{x_1,\dots,x_n\}$ be the final solution selected by the present-biased agent for~$I$, and let~$\w$ be the weight function on the set $\m{S}_I$ associated to~$I$. We have $\sum_{i=1}^n\w(x_i) = \alpha\,\opt(I)$. For every $i\in\{1,\dots,n\}$, let us denote by $\opt(I\smallsetminus\{x_1,\dots,x_i\})$ the cost of an optimal solution for the instance $I\smallsetminus\{x_1,\dots,x_i\}$, and by $S_{ \opt(I\smallsetminus\{x_1,\dots,x_i\})}$ a corresponding optimal solution. For $i=0$, $S_{ \opt(I\smallsetminus\{x_1,\dots,x_i\})}$ is an optimal solution for $I$. For $i=1,\dots,n$, we define 
$$
\bar{x}_i = \{x_i\}, 
\;
\mbox{and}
\;
\bar{y}_i=S_{ \opt(I\smallsetminus\{x_1,\dots,x_{i-1}\})}.
$$
Let $J$  be the instance with ground set 
$
\{\bar{x}_1,\dots,\bar{x}_n\}\cup\{\bar{y}_1,\dots,\bar{y}_n\},
$
and feasible solutions 
$$
\{\bar{y}_1\}, \{\bar{x}_1,\bar{y}_2\},\dots,\{\bar{x}_1,\dots,\bar{x}_{n-1},\bar{y}_n\}, \{\bar{x}_1,\dots,\bar{x}_n\}.
$$
Note that $\bar{x}_i \neq \bar{x}_j$ for every $i\neq j$, because $x_i\neq x_j$  for every $i\neq j$. Also, for every $i\in\{1,\dots,n-1\}$ and $j\in\{1,\dots,n\}$, $\bar{x}_i \neq \bar{y}_j$, because otherwise the sequence constructed by the time-biased agent for~$I$ would stop at~$x_k$  with $k<n$. Therefore, we have $J=J^{(n)}$, and, since $\w(x_i)\leq \opt(I)$ for every $i=1,\dots,n$, $n\geq \alpha$ holds. 

For analyzing the behavior of a present-biased agent with parameter~$\beta$ acting on~$J$, let us assume that $k$ steps were already performed by the agent, with $0\leq k < n$, resulting in constructing the partial solution $\{\bar{x}_1,\dots,\bar{x}_k\}$. (For $k=0$, this partial solution is empty). The feasible solutions for $J_k=J\smallsetminus \{\bar{x}_1,\dots,\bar{x}_k\}$ are 
\[
\{\bar{y}_1\},\dots, \{\bar{y}_{k+1}\}, \{\bar{x}_{k+1},\bar{y}_{k+2}\},\dots,\{\bar{x}_{k+1},\dots,\bar{x}_{n-1},\bar{y}_n\}, \{\bar{x}_{k+1},\dots,\bar{x}_n\}.
\]
Note that, for every $i\in\{1,\dots,n\}$, $\bar{\w}(\bar{x}_i)=\w(x_i)$, and  $\bar{\w}(\bar{y}_i)=\opt(I\smallsetminus\{x_1,\dots,x_{i-1}\})$. We claim that $\bar{x}_{k+1}$ is the next element chosen by the agent. Indeed, note first that $\bar{\w}(\bar{x}_{k+1})\leq \bar{\w}(\bar{y}_{k+2})$, as, otherwise, we would get $\w(x_{k+1}) > \opt(I\smallsetminus\{x_1,\dots,x_{k+1}\})$, contradicting the choice of $x_{k+1}$ by the agent performing on~$I$. As a consequence, 
$$
c_\beta(\{\bar{x}_{k+1},\bar{y}_{k+2}\})=\bar{\w}(\bar{x}_{k+1}) + \beta \,\bar{\w}(\bar{y}_{k+2}). 
$$
It follows from the above that, for every $j=1,\dots,k+1$, $c_\beta(\{\bar{y}_j\})\geq c_\beta(\{\bar{x}_{k+1},\bar{y}_{k+2}\})$, as the reverse inequality would contradict the choice of $x_{k+1}$ by the agent performing on~$I$. For the same reason, for every $\ell\in\{k+1,\dots,n-1\}$, and every $i\in\{k+1,\dots,\ell\}$, we have 
\[
c_\beta(\{\bar{x}_{k+1},\bar{y}_{k+2}\}) \leq \bar{\w}(\bar{x}_i)+\beta \Big (\sum_{j\in\{k+1,\dots,\ell\}\smallsetminus \{i\}} \bar{\w}(\bar{x}_j) +\bar{\w}(\bar{y}_{\ell+1})\Big)
\]
and
\[
c_\beta(\{\bar{x}_{k+1},\bar{y}_{k+2}\}) \leq \bar{\w}(\bar{y}_{\ell+1})+\beta \sum_{j\in\{k+1,\dots,\ell\}} \bar{\w}(\bar{x}_j).
\]
As a consequence, the present-biased agent performing on $J$ picks $\bar{x}_{k+1}$ at step~$k+1$, as claimed. The cost of the solution computed by the agent is $\sum_{i=1}^n\bar{\w}(\bar{x}_i)=\sum_{i=1}^n w(x_i)=\alpha\,\opt(I)$. On the other hand, by construction, $\opt(J)=\bar{\w}(\bar{y}_1)=\opt(I)$. The cost ratio of the solution computed by the agent for~$J$ is thus~$\alpha$, which completes the proof. 
\end{proof}

\section{Conclusion}
We demonstrated that, by focussing on present-biased agents solving tasks, specific detailed analysis can be carried on for each considered task, which enables to identify very different agent's behavior depending on the tasks (e.g., set cover vs. hitting set). Second, focussing on present-biased agents solving tasks enables to generalize the study to overestimation, and to maximization, providing a global picture of searching via present-biased agents. Yet, lots remain to be done for understanding the details of this picture. 

In particular, efforts could be made for studying other specific classical problems in the context of searching by a present-biased agent. This includes classical optimization problems like traveling salesman (TSP), metric TSP, maximum matching, feedback vertex set, etc. Such study may lead to a better understanding of the class of problems for which present-biased agents are efficient, and the class for which they act poorly. And, for problems for which present-biased agents are acting poorly, it may be of high interest to understand what kind of restrictions on the agent's strategy may help the agent finding better solutions. 

Another direction of research is further investigation of the influence of using approximation  algorithms by agents, as it is natural to assume that the agents are unable compute the cost exactly. We made some initial steps in this direction, but it seems that this area is almost unexplored. For example,
it can be noted that the upper bound for the cost ratio in Theorem~\ref{thm:min-betamoreone} can be rewritten under the assumption that the agent uses an $\alpha$-approximation algorithm. However, the bound gets blown-up by the factor $\alpha^s$, where $s$ is the size of the solution obtained by the agent (informally, we pay the factor $\alpha$ on each iteration). From the other side, the examples in Section~\ref{sec:hitting} show that this is not always so.    
Are there cases when this exponential blow-up unavoidable? The same question can be asked about maximization problems.

\end{document}